\documentclass[10pt]{amsart}

\usepackage{amsmath,amsthm,amssymb}
\usepackage{enumitem}
\usepackage{tikz}
\usetikzlibrary{cd}
\usepackage{xcolor}
\usepackage{mathrsfs}

\newcommand{\C}{\mathbb{C}}
\newcommand{\cA}{\mathcal{A}}

\newcommand{\cohshv}[1]{\widetilde{#1}}

\newcommand{\divides}{\mid}
\newcommand{\dual}[1]{{#1}^{\vee}}
\newcommand{\dualsigma}{\dual{\sigma}}
\newcommand{\faceof}{\preceq}

\newcommand{\gitquot}[2]{#1 // #2}
\newcommand{\glreg}[1]{\glsec{#1}{\strshf{#1}}}
\newcommand{\glsec}[2]{\Gamma\left( #1,\, #2 \right)}
\newcommand{\Gm}{\mathbb{G}_{\mathrm{m}}}
\newcommand{\inpr}[2]{\left\langle#1,\,#2\right\rangle}
\newcommand{\invelts}[1]{#1^{\times}}
\newcommand{\mProj}{\Proj_{\textrm{MH}}}
\newcommand{\MQ}{M_{\Q}}

\newcommand{\N}{\mathbb{N}}
\newcommand{\NQ}{N_{\Q}}
\newcommand{\NReal}{N_{\R}}

\newcommand{\plD}{\mathfrak{D}}
\newcommand{\Polg}[2]{\Polgen{}{#1}{#2}}
\newcommand{\Polgen}[3]{\Pol^{#1}_{#2}\!\left( #3 \right)}

\newcommand{\Polpg}[2]{\Polgen{+}{#1}{#2}}

\newcommand{\Polpq}{\Polpg{\sigma}{\NQ}}
\newcommand{\Polq}{\Polg{\sigma}{\NQ}}
\newcommand{\ppdivQ}{\ppdiv_{\Q}}

\newcommand{\Q}{\mathbb{Q}}
\newcommand{\Qpos}{\Q_{\geq 0}}

\newcommand{\R}{\mathbb{R}}
\newcommand{\restrict}[2]{\left.#1\right|_{#2}}
\newcommand{\set}[2]{\left\{#1\,\middle|\,#2\right\}}
\newcommand{\setl}[1]{\left\{\,#1\,\right\}}
\newcommand{\sizeof}[1]{\left|#1\right|}
\newcommand{\smstbl}[2]{#1_{\textup{\textrm{ss}}}(#2)}
\newcommand{\strshf}[1]{\mathscr{O}_{#1}}

\newcommand{\tvarieties}{\text{$T$-varieties}}

\newcommand{\Z}{\mathbb{Z}}

\DeclareMathOperator{\CDiv}{CDiv}
\DeclareMathOperator{\eval}{eval}
\DeclareMathOperator{\Hom}{Hom}
\DeclareMathOperator{\image}{image}

\DeclareMathOperator{\Pol}{Pol}

\DeclareMathOperator{\ppdiv}{PPDiv}

\DeclareMathOperator{\Proj}{Proj}
\DeclareMathOperator{\rank}{rank}

\DeclareMathOperator{\relint}{rel\,int}

\DeclareMathOperator{\Spec}{Spec}
\DeclareMathOperator{\tail}{tail}

\DeclareMathOperator{\WDiv}{WDiv}

\numberwithin{equation}{section}

\newtheorem{theorem}[equation]{Theorem}
\newtheorem{lemma}[equation]{Lemma}
\newtheorem{corollary}[equation]{Corollary}
\newtheorem{proposition}[equation]{Proposition}

\theoremstyle{definition}
\newtheorem{definition}[equation]{Definition}

\theoremstyle{remark}
\newtheorem{assumption}[equation]{Assumption}
\newtheorem{remark}[equation]{Remark}
\newtheorem{example}[equation]{Example}

\newtheorem{numpara}[equation]{}

\title[MHS and T-varieties]{Properties of Multihomogeneous Spaces and relation with T-varieties}
\author{Vivek Mohan Mallick}
\email{vmallick@iiserpune.ac.in}
\address{Department of Mathematics, IISER Pune, Dr Homi Bhabha Road, Pashan, Pune 411008, India}

\author{Kartik Roy}
\email{kartik.roy@students.iiserpune.ac.in}
\address{Department of Mathematics, IISER Pune, Dr Homi Bhabha Road, Pashan, Pune 411008, India}
\thanks{The second author is supported by CSIR.}



\begin{document}
\maketitle

\begin{abstract}
  We study multihomogeneous spaces corresponding to $\mathbb{Z}^n$-graded
  algebras over an algebraically closed field of characteristic $0$ and
  their relation with the description of $T$-varieties.
\end{abstract}

\section*{Introduction}

Algebraic varieties with torus actions, including but not limited to toric
varieties, have been at the centre of much attention for the past few
decades.  During our research we came across various constructions around
such varieties. This paper tries to relate two of these constructions under
some hypothesis.

The first object of interest is a variety with an effective
algebraic torus action. These were studied by various people, for example,
as toric varieties over discrete valuation rings considered by Kempf,
Knudsen, Mumford and Saint-Donat \cite{keknmusd:toroidalembeddingsi}; as a
part of the general case of varieties with the action of an reductive group
by Timash{\"e}v \cite{MR1470147}; and the case of $\C^*$ actions on normal
affine surfaces were studied by Flenner and Zaidenberg \cite{MR2020670} to
name a few. The theory was neatly generalized and combined into a single
theory by Altmann, Hausen and S{\"u}ss (see \cite{ah:affinetvar} for the
affine case and \cite{ahs:gentvar} for the general case). The combinatorial
descriptions of the geometric properties were studied extensively and are
reported in the survey \cite{aipsv:geomtvar}.  There has been quite a bit of
activity in this area in the recent years.

Another concept which drew our attention was that of a multihomogeneous
projective space defined by Brenner and Schr{\"o}er
\cite{breshr:ampfamilies}. These spaces
are generalizations of weighted projective spaces and are divisorial
schemes. Brenner and Schr{\"o}er gave a criterion for a scheme of a finite
type over a noetherian ring to be divisorial in terms of existence of an
embedding of the scheme into a multihomogeneous space associated to a
multigraded algebra \cite[corollary 4.7]{breshr:ampfamilies}. Extending
their work, Zanchetta \cite{zanchetta:embdivsch} proved that the ambient
multihomogeneous space can be chosen to be smooth. Some applications of this
theory can be seen in Kanda \cite{MR4033823}.

This paper delves into the relationship between these two concepts. Digging a
bit deeper, not surprisingly, GIT quotients play a role in both the
theories. We try to follow this link as far as we could.

While studying and working with multihomogeneous spaces we proved some
results generalizing similar results in weighted projective spaces (see, for example, \cite{MR704986} and \cite{MR3307753}). A
criterion for a twisted module, defined in a similar fashion as the twisted
modules on projective varieties, to be a line bundle (theorem
\ref{the:criOXdlb}).  Furthermore, in multihomogeneous spaces, the points
need not correspond to homogeneous prime ideals. This paper proves a
criterion for this to happen (corollary \ref{cor:ptsarepr}).

Normal varieties along with an effective action of a torus $T$ are called
$T$-varieties. Such varieties can be described by partially combinatorial data in
the form of a semiprojective variety $Y$ and a proper polyhedral divisor on
$Y$, which are generalization of usual $\Q$-divisors on $Y$ where rational
linear combinations are replaced by formal sums of the divisors with
polyhedral coefficients.

We show that $Y$ associated to an affine $T$-variety $X = \Spec A$ is
birational to a multihomogeneous space obtained as the Proj of $\Hom(T,
\Gm)$-graded ring $A$, where the grading is obtained by taking isotypical
components under the torus action (see theorem \ref{the:birYmHSA}). We end
the paper by giving one criterion when this birational morphism is an
isomorphism.

The paper is divided into 4 sections. Starting with a review of the theory
of multihomogeneous spaces, the first section goes on to study some
conditions under which the points in the multihomogeneous spaces correspond
to homogeneous prime ideals. This is not true in general as remarked in
\cite[remark 2.3]{breshr:ampfamilies}. We end the section with a condition
under which the multihomogeneous Proj will be normal.

The second section is a review of the theory of \tvarieties{}. This section
is just for clarity of exposition and fixing notation and does not contain
any new results.

The third section defines and proves some results for twisted sheaves over
multihomogeneous spaces. We end with some hypothesis under which the twisted
sheaves are line bundles.

The last section deals with the question about when these constructions yield
the same space. After studying some cases where this fails, we end
with a theorem which describes some sufficient conditions under which they
are isomorphic.

\noindent\textbf{Acknowledgement:} Both the authors thank IISER Pune for its
hospitality where all the work was done. The second author thanks CSIR for
funding his research.

\section{Multihomogeneous Spaces}

In this section we review the theory of multihomogeneous spaces. For more
details regarding multihomogeneous spaces we refer to \cite[section
2]{breshr:ampfamilies}. See also \cite{MR3307753} for results on geometry of multigraded algebras and their properties.

\begin{definition}
  \label{def:prdrlelt}
  Let $D$ be a finitely generated abelian group and
  \begin{equation*}
  A = \bigoplus_{d \in D} A_d 
  \end{equation*}
  be a $D$-graded ring. One says
  that $A$ is \emph{periodic} if $D' = \set{d \in D}{\exists f \in A_d \cap
  \invelts{A}}$, the subgroup of $D$ consisting of degrees of all the
  homogeneous invertible elements in $A$ is a finite index subgroup.  A
  homogeneous element $f$ in a $D$-graded ring $A$ is said to be
  \emph{relevant} if $A_f$ is periodic.  For a relevant element $f$, note
  that the localization $A_f$ is $D$-graded. We shall denote the degree $0$
  part of $A_f$ by $A_{(f)}$.
\end{definition}

The following lemma by Brenner and Schr\"oer is useful.

\begin{lemma}( \cite{breshr:ampfamilies}, lemma 2.1) \label{lem:prdrspec}
  Let $D$ be a finitely generated abelian group and \begin{equation*} A =
  \bigoplus_{d \in D} A_d \end{equation*} be a $D$-graded periodic ring. Then the projection $\Spec A \longrightarrow \Spec (A_{0})$ is a geometric quotient in GIT sense. 
\end{lemma}

\begin{definition}
  \label{def:muhospcs}
  For $D$ and $A$ as in definition \ref{def:prdrlelt}, the grading on
  $A$ corresponds to an action of the diagonalizable group scheme $\Spec
  A_0[D]$ on $\Spec A$. Let $Q$ be the quotient in the category of ringed
  spaces. Now for a relevant element $f$, consider the inclusion
  \begin{equation*}
    D_+(f) = \Spec A_{(f)} \subset Q.
  \end{equation*}
  One defines
  \begin{equation*}
    \mProj A = \bigcup_{
      \substack{
	f \in A \\
	f \text{ is relevant}
      }
    }
    D_+(f) \subset Q.
  \end{equation*}
\end{definition}

\begin{remark}
  \label{rem:ptsinmhs}
  The points in a multihomogeneous projective space $\mProj A$ of a
  $D$-graded ring $A$ correspond to homogeneous ideals in A which may not be
  prime (see \cite[remark 2.3]{breshr:ampfamilies}). However, these ideals
  have the property that all the homogeneous elements in the complement form
  a multiplicatively closed set.
\end{remark}

\begin{proposition} \label{pro:prchompr}
  Suppose $D$ is a free finitely generated $\Z$-module and $A = \bigoplus_{d \in D} A_d$ is a $D$-graded ring. Assume that we have a
  collection of relevant elements $F$ such that
  \begin{equation*}
    \mProj A = \bigcup_{f \in F} \Spec A_{(f)}
  \end{equation*}
  and for each $f \in F$, $\set{d \in D}{d = \deg g \text{ for some homogeneous } g \in
  \invelts{A_f}} = D$. Then every point $p \in \mProj A$ corresponds to a
  homogeneous prime in $A$.
\end{proposition}
\begin{proof}
  Suppose $p \in \Spec A_{(f)}$ for some relevant element $f \in A$. Then
  $A_f$ is periodic and
  \begin{equation*}
    D' = \set{d \in D}{d = \deg g \text{ for some homogeneous } g \in \invelts{A_f}}
  \end{equation*}
  is a free subgroup of $D$ of finite index. Define
  \begin{equation*}
    A_f' = \bigoplus_{d \in D'} (A_f)_d.
  \end{equation*}
  It is easy to see that in this case, $A_f' = A_{(f)}[T_1^{\pm 1}, \dotsc,
  T_r^{\pm 1}]$, where $r = \rank D'$.

  Note the primes $P \in A_{(f)}$ correspond to the primes $P[T_1^{\pm 1},
  \dotsc, T_r^{\pm 1}] \subset A_f'$. Now consider the diagram
  \begin{equation*}
    \begin{tikzcd}
      & & A \arrow[d, hook'] \\
      A_{(f)} \arrow[r, hook] & A_f' \arrow[r, hook] & A_f
    \end{tikzcd}
  \end{equation*}
  It is easy to see that if $A_f' = A_f$, then the primes in $A_{(f)}$ would
  correspond to homogeneous primes in $A$ which do not contain $f$. The
  condition $A_f' = A_f$ holds whenever the hypothesis of the proposition is
  satisfied.
\end{proof}

\begin{corollary}
  \label{cor:hprAprAf}
  Under the hypothesis of proposition \ref{pro:prchompr}, the points in
  $D_+(f) \subset \mProj A$ correspond to all homogeneous primes in
  $A$ which do not contain $f$.
\end{corollary}

\begin{proof}
  This was mentioned in the proof of proposition \ref{pro:prchompr} after
  the diagram.
\end{proof}

\begin{corollary}
\label{cor:ptsarepr}
  Suppose $A$ is a $D$-graded ring generated over $A_0$ by a set
  \[
  \setl{a_1, \dotsc, a_n}
  \]
  of homogeneous elements such that any $\Z$-linearly independent subset of 
  \[
  \setl{\deg a_1, \dotsc, \deg a_n}
  \]
  having $\rank D$ elements is a basis for the abelian
  group $D$. In this case the hypothesis of proposition \ref{pro:prchompr}
  holds and hence the points in $\mProj A$ will
  correspond to homogeneous prime ideals in the graded ring $A$.
\end{corollary}

\begin{remark}
  \label{rem:prjempty}
  The way $\mProj A$ is defined for a $D$-graded ring $A$, it can happen
  that $A$ has no relevant element and then $\mProj A = \emptyset$. If
  $A$ is a finitely generated algebra over $A_0$, one sufficient condition
  for the existence of relevant elements is that there exists a collection
  of homogeneous generators $\set{x_i}{1 \leq i \leq r}$ such that
  $\set{\deg x_i}{1 \leq i \leq r}$ generates a finite index subgroup in
  $D$. This condition is easy to check, for example, when $A$ is the
  polynomial ring over $\C$.
\end{remark}

\begin{remark}
  \label{rem:frelgmqt}
  By \cite[Lemma 2.1]{breshr:ampfamilies},  the map $\Spec A_f \longrightarrow \Spec A_{(f)}$, which is induced by the inclusion $A_{(f)}
  \hookrightarrow A_f$, is a geometric quotient.
\end{remark}

By definition, the collection of affine open subschemes
\[
\set{D_+(f)}{f \in A \text{ is homogeneous and relevant}}
\]
covers $\mProj A$. We state the following easy fact for subsequent use.

\begin{lemma}
  \label{lem:dpfdpgfg}
  With the notation as above, $D_+(f) \cap D_+(g) = D_+(fg) \subset \mProj
  A$.
\end{lemma}

\begin{proof}
  This is implicit in \cite[propostion 3.1]{breshr:ampfamilies}. Note that
  for relevant elements $f$ and $g$ in $A$, $\Spec A_{fg} = \Spec A_f \cap
  \Spec A_g$ as subschemes of $\Spec A$. Now $\Spec A_{(fg)}$, $\Spec
  A_{(f)}$ and $\Spec A_{(g)}$ are geometric quotients (see remark
  \ref{rem:frelgmqt}) under the action of $\Spec A_0[D]$ and hence $\Spec
  A_{(fg)} = \Spec A_{(f)} \cap \Spec A_{(g)}$ considered as a subscheme of
  $\mProj A$.
\end{proof}

For later, we record two results of Brenner and Schr{\"o}er regarding
finiteness.

\begin{lemma}[\cite{breshr:ampfamilies}, lemma 2.4]
  \label{lem:noethcnd}
  For a finitely generated abelian group $D$ and a $D$-graded ring
  $A$, the following are equivalent:
  \begin{enumerate}[label=(\textit{\roman*})]
    \item The ring $A$ is noetherian.
    \item $A_0$ is noetherian and $A$ is an $A_0$-algebra of finite type.
  \end{enumerate}
\end{lemma}

\begin{proposition}[\cite{breshr:ampfamilies}, proposition 2.5]
  \label{pro:unclfntp}
  Suppose $A$ is a noetherian ring graded by a finitely generated abelian
  group $D$. Then the morphism $\varphi \colon \mProj A \longrightarrow
  \Spec A_0$ is universally closed and of finite type.
\end{proposition}
 \begin{definition} [\cite{breshr:ampfamilies}, page 10]\label{def:simtorembed}
 Let $R$ be a ring, $M$ be a free abelian group of finite rank, and $N:= \Hom(M,\Z)$ be dual of $M$. Let $X$ be an $R-$scheme and $T:=\Spec R[M]$ be the torus. A simplicial torus embedding of torus $T$ is $T-$equivariant open map $T \hookrightarrow X$ locally given by semigroup algebra homomorphisms $R[\dual{\sigma} \cap M] \rightarrow R[M]$, where $\sigma$ is a strongly convex, simplicial cone in $N$.
 \end{definition}
 
 \begin{remark}
   If $X$ is a toric variety with torus $T$, then $X$ is a simplicial torus embedding of the torus $T$. There are other schemes which are simplicial torus embeddings of some torus. Homogeneous spectrum of miltigraded polynomial algebras are examples of this type.
 \end{remark}
 
 Let $D$ be an abelian group of finite type and $A = k[x_1, \dots, x_n]$ be a $D-$graded polynomial $k-$algebra. Suppose the grading is given by a linear map $P: \Z^{n} \rightarrow D$ with finite co-kernel. Then we have the following sequence of abelian groups
 \begin{equation*}
     0 \rightarrow M \rightarrow Z^{n} \rightarrow D,
 \end{equation*}
  where $M$ is the kernel of $P$.
  
  \begin{proposition}[\cite{breshr:ampfamilies}, proposition 3.4]
  Assume the above setting. Then $\mProj A$ is a simplicial torus embedding of the torus $\Spec k[M]$.
  \end{proposition}
  
  \begin{remark}[\cite{breshr:ampfamilies}, remark 3.7]\label{rem:relconecorr}
    Again we assume the above setting. Let $I = \{1, \dots, n\}$ be an index set and $N:= \Hom(M,\Z)$ be the dual of $M$. Let $\text{pr}_i: \Z^{n} \rightarrow \Z, i \in I$ be projections. We associate each subset $J \subset I$ to cone $\sigma_J \subset \NReal$ generated by $\text{pr}_i|_{M}, i \in J$. Then we have a correspondence between subsets $J$ of $I$ with $\prod_{i \in J}x_i$ relevant, and strongly convex, simplicial cones $\sigma_{I\setminus J} \subset \NReal$.
  \end{remark}

\begin{proposition}\label{pro:normal}
Suppose $A$ is a noetherian normal ring satisfying above hypothesis. Then $\mProj A$ is a normal scheme.
\end{proposition}

\begin{proof}
  It is enough to check normality over an affine open cover $\{D_{+}(f_i)\}$ of
  $\mProj A$. Let $f$ be a relevant element of $A$, $H$ be the set of
  nonzero homogeneous elements in $A_f$, $k(A_f)$ and $ k(A_{(f)})$ be
  function fields of $A_f$ and $A_{(f)}$ respectively. Then $H^{-1}(A_f)$ is
  a graded ring and $A_f \hookrightarrow H^{-1}(A_f)$ is a graded
  monomorphism.  Furthermore, we have the inclusion of the $0$-th component
  of the grading: $k(A_{(f)}) \subset \big(H^{-1}(A_f)\big)_{0}$. 
   \begin{equation*}
    \begin{tikzcd}
      & k(A_f) \\
      A_f \arrow[r, hook] \arrow[ur, hook] &  (A_f)_{H} \arrow[u, hook] \\
      A_{(f)} \arrow[r, hook] \arrow[u]  & k(A_{(f)})\arrow[u] 
    \end{tikzcd}
  \end{equation*}
  Since $A$ is normal, we have $A_f$ is normal in $k(A_f)$ and therefore
  integrally closed in $(A_f)_{H}$. We want to show that $A_{(f)}$ is
  normal. Let $a \in k(A_{(f)})$ be integral over $A_{(f)}$. Then $a$ is
  integral over $A_f$ and hence $a \in (A_f)_{0} = A_{(f)}$ since $a \in
  \big(H^{-1}(A_f)\big)_{0}$.
\end{proof}

\section{Affine \tvarieties}
\label{sec:afftvars}

In this section we shall recall Altmann and Hausen's theory of affine
$T$-varieties. The results from this paper are summarized here for ready
reference. In this section (cf.\ \cite[section 2]{ah:affinetvar}),
\emph{
  all algebraic varieties are integral schemes of finite type over an
  algebraic closed field $k$ of characteristic zero.
}


\begin{definition}
  \label{def:tvariety}
  (See, \cite[section 1.1]{aipsv:geomtvar}) Suppose $M$ is a free
  $\Z$-module of finite rank and $T = \Spec k[M]$ be the corresponding
  torus. An \emph{affine $T$-variety} is a normal affine variety with an
  effective action of $T$.
\end{definition}

The $T$-varieties have a partial combinatorial description which we review
below.


\begin{definition}
  \label{def:ppdivonY}
  Suppose $Y$ is a semiprojective variety; i.e.\ an algebraic variety such
  that the $k$-algebra  $\glreg{Y}$ is finitely generated and $Y$ is
  projective over $Y_0 = \Spec \glreg{Y}$. Let $N$ be a finite rank free
  $\Z$-module and $M = \Hom_{\Z}(N, \Z)$. Then $T = \Spec k[M]$ is a split
  torus over $k$. Fix a pointed (i.e.\ a strongly convex polyhedral) cone
  $\sigma \in \NQ = N \otimes_{\Z} \Q$. Let $\Polpq$ be the collection of
  convex polyhedra $\Delta$ (i.e.\ finite intersection of closed half
  spaces) such that
  \begin{align*}
    \tail(\Delta) &= \set{v \in \NQ}{v' + tv \in \Delta\ \forall v' \in
    \Delta, t \in \Qpos} \\
    &= \sigma,
  \end{align*}
  and let the \emph{group of $\sigma$-polyhedra}, $\Polq$, be the
  Grothendieck group of $\Polpq$ (see \cite[definition 1.2]{ah:affinetvar}).
  The \emph{group of rational polyhedral Weil} (respectively,
  \emph{Cartier}) \emph{divisors} with respect to $\sigma$ is defined as
  $\WDiv_{\Q}(Y, \sigma) := \Polq \otimes_{\Z} \WDiv(Y)$ (respectively,
  $\CDiv_{\Q}(Y, \sigma) := \Polq \otimes_{\Z} \CDiv(Y)$). To describe the
  integral polyhedral divisors, one considers those polyhedra which admit a
  decomposition as a Minkowski sum of a polytope with vertices in $N$ and
  $\sigma$ (see \cite[definitions 1.1, 1.2 and 2.3]{ah:affinetvar}). For an
  element $u \in \dualsigma$, one can define a linear \emph{evaluation}
  functional $\eval_u \colon \Polq \longrightarrow \Q$ such that for any
  $\Delta \in \Polpq$, $\eval_u(\Delta) = \min_{v \in \Delta} \inpr{u}{v}$.
  A Weil (respectively, Cartier) polyhedral divisor is an element of $\Polq
  \otimes_{\Z} \WDiv(Y)$ (respectively, $\Polq \otimes_{\Z} \WDiv(Y)$).
  Given a polyhedral divisor $\plD = \sum_D \Delta_D \otimes D$, and an $u
  \in \dualsigma$, one defines $\plD(u) = \sum_D \eval_u(\Delta_D) D$.  By a
  \emph{pp-divisor} (or a \emph{proper, polyhedral divisor}) one means a
  polyhedral divisor $\plD \in \CDiv_{\Q}(Y, \sigma)$ such that it can be
  represented as $\plD = \sum_{i=1}^{r} \Delta_i \otimes D_i$ with $\Delta_i
  \in \Polpq$ and effective divisors $D_i$ satisfying the following: for any
  $u \in \relint \dualsigma$, $\plD(u)$ is a big divisor on $Y$; and for any
  $u \in \dualsigma$, $\plD(u)$ is semiample (see \cite[definition
  2.7]{ah:affinetvar}). The semigroup of all pp-divisors having tail cone
  $\sigma$ is denoted by $\ppdivQ(Y, \sigma)$.
\end{definition}

\begin{definition}[Weight cone]
  \label{def:wghtcone}
  (See \cite[introductory discussion, section 3]{ah:affinetvar},
  \cite[section 2]{bh:gittoric}.)
  Given an affine variety $X = \Spec A$ with an effective action of the
  torus $T = \Spec k[M]$, suppose that the decomposition of $A$ into $\chi^m
  \colon T \to \Gm(k)$ semi-invariants is given by $A = \bigoplus_{m \in M}
  A_m$. Then the \emph{weight cone} is the convex polyhedral cone $\omega
  \subset \MQ$ generated by the \emph{weight monoid} $S = \set{m \in M}{A_m
  \neq \setl{0}}$.
\end{definition}

Altmann and Hausen prove the following theorem.

\begin{theorem}[AH08, Theorem 3.1 and 3.4]
  \label{the:affTplDY}
  Given a normal, semiprojective variety $Y$, a lattice $N$, the dual
  lattice $M$, a pointed cone $\sigma \subset \NQ$, a pp-divisor $\plD \in
  \ppdivQ(Y, \sigma)$, the affine scheme associated to $(Y, \plD)$ is
  described as 
  \begin{equation*}
    X = \Spec \glsec{Y}{\bigoplus_{u \in \dualsigma \cap M}
    \strshf{Y}(\plD(u))}.
  \end{equation*}
  Then $X$ is a normal $T$-variety where $T = \Spec k[M]$. Moreover given
  any normal affine $T$-variety $X = \Spec A$ with weight cone $\omega \in
  \MQ$, there is exists a normal semiprojective variety $Y$ and a pp-divisor
  $\plD \in \ppdivQ(Y, \dual{\omega})$ such that the $T$-variety associated
  to $(Y, \plD)$ is $X$.
\end{theorem}

\begin{remark}
\label{rem:nonuniqueY}
  Note that the variety $Y$ in theorem \ref{the:affTplDY} is not uniquely determined, but is 
  unique up to a birational class. For more details, see \cite[corollary 8.12]{ah:affinetvar}.
\end{remark}

\begin{remark}
  \label{rem:tilXandX}
  Suppose $Y$ and $\plD$ are as in the first part of the theorem
  \ref{the:affTplDY}. Let $\cA = \bigoplus_{m \in \dualsigma} \cA_m$ where
  $\cA_m = \strshf{Y}(\plD(m))$. Then, one can also consider the relative
  spectrum $\tilde{X} = \Spec_Y \cA$. Then $A = \Gamma(Y, \cA)$.
  The scheme $\tilde{X}$ is a normal affine variety with an effective
  $T$ action such that $\pi \colon \tilde{X} \longrightarrow Y$ is a good
  quotient. Furthermore there is a contraction morphism $r \colon \tilde{X}
  \longrightarrow X$ which is proper, birational and $T$-equivariant.
\end{remark}

The orbits are described using orbit cones which in turn define a GIT fan. We
shall require this concept later on and hence we recall the associated definitions briefly.

\begin{definition}
  \label{def:ormndetc}
  (See \cite[definition 5.1]{ah:affinetvar}, \cite[definition
  2.1]{bh:gittoric}.)
  Let $X = \Spec A$ be a normal affine variety with an action of a torus $T
  = \Spec k[M]$. Suppose the action of the torus determines the
  decomposition $A = \bigoplus_{m \in M} A_m$ into spaces of semi-invariants.
  For $x \in X$, the \emph{orbit monoid} is the submonoid $S(x) \subset M$
  defined as
  \begin{equation*}
    S(x) = \set{m \in M}{\exists f \in A_m \text{ such that } f(x) \neq 0}.
  \end{equation*}
  The orbit monoid generates a convex cone $\omega(x) \subset \MQ$ called
  the \emph{orbit cone}.

  The set of $\chi^m$ semistable points is defined as
  \begin{equation*}
    \smstbl{X}{m} = \set{x \in X}{m \in \omega(x)}
  \end{equation*}
  The \emph{GIT cone} associated to $m \in \omega \cap M$ is the
  intersection $\lambda(m) := \bigcap_{x \in X; m \in \omega(x)}
  \omega(x)$. Suppose $\omega$ is the weight cone for the torus action on $X$.
  The collection of GIT cones $\Lambda =
  \set{\lambda(m)}{m \in \omega \cap M}$ forms a quasi-fan in $\MQ$ having
  $\omega$ as its support. For brevity, we shall call this quasifan as a
  \emph{GIT fan}.
\end{definition}


Given a normal variety $X = \Spec A = \Spec \bigoplus_{m \in M} A_m$ with an
effective torus action, theorem \ref{the:affTplDY} above ensures the
existence of $(Y, \plD)$. We recall the description of $Y$, as it will be
useful in section \ref{sec:relnmhst}. According to the theory in \cite[section
2]{bh:gittoric}, $\smstbl{X}{m} = \smstbl{X}{m'}$ for $m, m'$ belonging to
the relative interior of a GIT cone $\lambda$ of the GIT fan $\Lambda$. Let
$X_{\lambda} := \smstbl{X}{m}$ for some $m \in \relint \lambda$. Then one
also has that $Y_m = \gitquot{\smstbl{X}{m}}{T} = \Proj \bigoplus_{r \in \Z}
A_{rm}$. Thus, $Y_m$'s also depend only on the fan $\lambda$ such that
$m \in \relint \lambda$, and hence are denoted by $Y_{\lambda}$. If
$\lambda' \faceof \lambda$, then one has a birational morphism
$\varphi_{\gamma \lambda} \colon Y_{\lambda} \longrightarrow
Y_{\gamma}$. Putting everything together compatibly one has the following
diagram which also defines $Y$ (see \cite[section 6]{ah:affinetvar}):
\begin{equation}
  \label{equ:constrcY}
    \begin{tikzcd}
      & X' \arrow[r]
      \arrow[d]
      & X_{\lambda} \arrow[r] \arrow[d]
      & X_{\gamma} \arrow[d] \arrow[r]
      & X \arrow[dd]
      \\ Y  \arrow[r] \arrow[drrrr, bend right=15] \arrow[d, equal]
      & Y' \arrow[r] \arrow[drrr, bend right=10]
      & Y_\lambda \arrow[r] \arrow[drr, bend right=5]
      & Y_\gamma \arrow[dr]
      &
      \\ \textrm{Normal}(\overline{\image(X')})
      & 
      &
      &
      & Y_{0}
    \end{tikzcd}
\end{equation}
where $X' = \varprojlim X_{\lambda}$ and $Y' = \varprojlim
Y_{\lambda}$. It is also known that $Y$ is a good quotient of the torus
action on $X$.

In the previous paragraph, we constructed the $Y$ in the pair $(Y, \plD)$
describing affine variety $X$ with an effective action of the torus. The
construction of the pp-divisor $\plD$ is not relevant to this paper.

\section{Some results about Multihomogeneous spaces}

\subsection{Sheaves associated to multigraded modules}
Consider a finitely generated abelian group $D$ and let $A$ be a $D$-graded
ring. Suppose $M = \bigoplus_{d \in D} M_d$ is a $D$-graded $A$-module. Just
as in the case of quasicoherent sheaves of modules over $\Proj$ of a $\N$-graded ring
\cite[definition before proposition 5.11, page 116]{hartshorne:alggeom},
we can construct $\cohshv{M}$. We sketch some details to fix notation
and show the similarities in the two setups.

In the construction of $\Proj$ of a $\Z$-graded module over a $\N$-graded
ring, by the definition of $\Proj A$ the points correspond to homogeneous
prime ideals in the defining graded ring which do not contain the whole of
the irrelevant ideal. However this is no longer true for a multihomogeneous
space and a point $P$ may correspond to an ideal which is not a prime.

Let $A$ be a $D$-graded ring and $M$ be a $D$-graded coherent module on $A$ with the usual
condition that $A_d M_{d'} \subset M_{d + d'}$. Since the points $p$ in
$\mProj A$ correspond to graded ideals $I_p$ in $A$ such that the homogeneous
elements in the complement $A \setminus I_p$ form a multiplicatively closed
set, it is still true that the stalk of the structure sheaf at $p$ is given
by $A_{(I_p)}$ (see remark \ref{rem:ptsinmhs}). One can now define
$\widetilde{M}$ in the same way by associating to $U \subset \mProj A$, the
$\strshf{\mProj A}(U)$-module of sections $s \colon U \to \coprod_{p \in U}
M_{(I_p)}$ satisfying the usual condition that locally such $s$ should be
defined by a single element of the form $m / a$ with $m \in M$ and
$a \in A$ but not in any of the ideals $I_p$. These modules are coherent
under some mild conditions, as we state below. Note that, given a $D$-graded
$A$-module $M = \bigoplus_{d \in D} M_d$ and an $e \in D$, one can define a
graded module $M(e)$ where as $A$-modules $M(e) = M$, but $M(e)_d =
M_{d + e}\ \forall d \in D$.

\begin{lemma}
  \label{lem:cohshvms}
  Suppose $D$ is a finitely generated abelian group and $A$ is a $D$-graded
  integral noetherian ring. Then for $X = \mProj A$, the following hold
  \begin{enumerate}[nosep, label=(\textit{\alph*})]
    \item \label{ite:AtldisOX}
      $\widetilde{A} = \strshf{X}$. This allows us to define
      \begin{equation*}
	\strshf{X}(d) := \widetilde{A(d)}.
      \end{equation*}
      $\strshf{X}(d)$ is a coherent sheaf.
    \item \label{ite:MDpfMftd}
      For a  $D$-graded $A$-module $M$, $\widetilde{M}$ is
      quasi-coherent and $\restrict{\widetilde{M}}{D_+(f)} \cong
      \widetilde{M_{(f)}}$ for any relevant element $f \in A$, where $\widetilde{M_{(f)}}$
      is the sheaf of modules over $\Spec A_{(f)}$ corresponding to the module $M_{(f)}$,
      the degree zero elements in $M_f$. Moreover, $\widetilde{M}$ is coherent whenever $M$ is finitely generated.
      \item The functor  $M \rightarrow \widetilde{M}$ is an covariant exact functor from category of $D$-graded $A$-modules to category of quasi-coherent $\strshf{X}$-modules, and commutes with direct limits and direct sums. 
  \end{enumerate}
\end{lemma}

The proof follows almost by definition and is very similar to proof of
\cite[proposition 5.11]{hartshorne:alggeom}.  The proof of the next lemma is
also evident.
\begin{remark}
Note that we have used $M$ as a lattice as well as $A$-module. Meaning should be clear from context.
\end{remark}
\begin{remark}
In general, the functor $\widetilde{*}$ is not faithful, even for projective varieties.
\end{remark}

\begin{lemma}
  \label{lem:fghypths}
  Suppose $D$ is a finitely generated abelian group and $A$ is a $D$-graded
  algebra such that $A = A_0[x_1, \dotsc, x_r]$, where $x_i \in A_{d_i}$ are
  homogeneous. Then $\set{d \in D}{A_d \neq 0}$ generate a finite index
  subgroup of $D$ if and only if $\set{d_i}{1 \leq i \leq r}$ does.
\end{lemma}

This lemma provides a way to ensure one of the points of the hypothesis in
the theorem below.

\begin{theorem}
  \label{the:glsecdAd}
  Suppose $D$ is a \emph{free} finitely generated abelian group and $A =
  \bigoplus_{d \in D} A_d$ is a $D$-graded \emph{integral domain} which is
  finitely generated by homogeneous elements $x_1, \dotsc, x_r \in A$ over
  the ring $A_0$. Also assume that for all $k$, $1 \leq k \leq r$, the set
  $\set{\deg x_i}{1 \leq i \leq r, i \neq k}$ generates a finite index subgroup of $D$. Let $X = \mProj
  A$. Then $\Gamma(X, \strshf{X}(d)) \cong A_d$. Furthermore,
  $\strshf{X}(d)$ is a reflexive sheaf.
\end{theorem}

Before proving the theorem, we observe a fact.

\begin{lemma}
  With the notation as in theorem \ref{the:glsecdAd},
  \label{lem:dplmonlX}
  \begin{equation*}
    X = \mProj A = \bigcup_{
      \substack{
	f \colon \text{ is relevant and } \\
	\text{ is a monomial in } x_1, \dotsc, x_r
      }
    } D_+(f).
  \end{equation*}
\end{lemma}

\begin{proof}
  We shall prove this for $D_+(f)$ for every relevant $f$ and the lemma
  will follow. Suppose $f = m_1 + \dotsb + m_t$ where each $m_i$ is a
  monomial. Any point $p$ in $D_+(f)$ corresponds to a homogeneous ideal $P$ in $A$
  such that the set of 
  homogeneous elements in $A \setminus P$ is
  multiplicatively closed. Let $H$ be the collection of all such homogeneous ideals.
  \begin{equation*}
    D_+(f) = \set{P \in H}{f \notin P} \subset \bigcup_{i=1}^t
    \set{P \in H}{m_i \notin P} = \bigcup_{i=1}^t D_+(m_i)
  \end{equation*}
  as was to be proved.
\end{proof}

Now we return to the proof of the theorem.

\begin{proof}[Proof of theorem \ref{the:glsecdAd}]
  Giving an element $t \in \Gamma(X, \strshf{X}(d))$ is the same as giving a
  collection $t_f \in D_+(f) = \Spec A_{(f)}$ for each relevant monomial $f$
  such that they agree on the pairwise intersections: $D_+(f) \cap D_+(g) =
  D_+(fg)$ (see lemma \ref{lem:dpfdpgfg}).

  Suppose $t \in \Gamma(X, \strshf{X}(d))$. For each relevant monomial $f
  \in A$ (which are enough to consider by lemma \ref{lem:dplmonlX}),
  \begin{equation*}
    \restrict{t}{D_+(f)} \in \strshf{X}(d)\biggl( D_+(f) \biggr) =
    \widetilde{A(d)}(D_+(f)) = (A_f)_d,
  \end{equation*}
  the $d$-th component of the $D$-graded ring $A_f$. Thus, for each such
  $f$ write
  \begin{equation*}
    \restrict{t}{D_+(f)} = \frac{p_f}{f^{k_f}}
  \end{equation*}
  where $\deg p_f - k_f \deg f = d$.  Now since $A$ is a domain, each $A_f
  \subset A_{x_1 \dotsm x_r}$, and since the local expressions of $t$ match over
  the intersections, $t$ is of the form $x_1^{\alpha_1} \dotsm
  x_r^{\alpha_r} f'$ with $f' \in A$. Since for each $i$, $x_1 \dotsm
  \hat{x_i} \dotsm x_r$ is relevant, $x_1^{\alpha_1} \dotsm
  x_r^{\alpha_r} f' \in A_{x_1 \dotsm \hat{x_i} \dotsm x_r}$ implies that
  $\alpha_i \geq 0$. This proves that $t \in A$ and therefore, $t \in A_d$.

  Since $\Hom_{\strshf{X}}(\strshf{X}(d), \strshf{X}) = \strshf{X}(-d)$ for
  all $d \in D$, reflexivity of $\strshf{X}(d)$ is clear.
\end{proof}

\begin{example}
  The hypothesis of the above theorem is necessary. For example, consider the ring
  $A = \C[X, Y, Z]$ with $\Z^2$-grading given by
  \begin{align*}
      \deg X &= (0, 1) & \deg Y &= (1, 0) = \deg Z
  \end{align*}
  The scheme $\mProj A$ is covered by two affines $D_+(XY)$ and $D_+(XZ)$. Now consider
  the module $M = A((2, -1))$. Consider the section $YZ / X$ which is defined over
  both $\widetilde{M}(D_+(XY))$ and $\widetilde{M}(D_+(XZ))$. Therefore, 
  $YZ / X \in \Gamma(\mProj A, \widetilde{M})$, whereas $A_{(2, -1)} = 0$.
\end{example}

\subsection{Line bundles on Multihomogeneous spaces}

The reflexive coherent sheaves of modules $\strshf{X}(d)$ will not be line bundles for
every $d \in D$. We give a criterion for these to be line bundles
generalizing the well-known similar results for weighted projective spaces.
Before that we prove a short lemma.

\begin{lemma}
  \label{lem:canchsem}
  Suppose $A$ is a $D$-graded ring for a finitely generated free abelian
  group $D$, generated as an $A_0$-algebra by homogeneous elements $x_1,
  \dotsc, x_r$. Suppose $A^{\times} = A_0^{\times}$.
  Assume that $f$ is a relevant
  monomial in $A$.  Suppose $d \in D_f$, where $D_f$ is the sublattice of $D$
  generated by
  \[
  \set{\deg a}{a \text{ divides } f^N \text{ for some } N > 0}.
  \]
  Then there is a monomial $m$ in $x_1, \dotsc,
  x_r$ and $k \in \N \cup \setl{0}$ such that $\deg(m / f^k) = d$ and $m \divides f^N$ for
  some $N > 0$.
\end{lemma}

\begin{proof}
  Suppose $f = x_{i_1}^{\alpha_1} \dotsm x_{i_s}^{\alpha_s}$. Then $D_f$ is generated
  by $\setl{\deg x_{i_1}, \dotsc, \deg x_{i_s}}$. Then for $d \in D_f$, there exists
  integers $a_1, \dotsc, a_s$ such that $d = \sum_{j=1}^s a_j d_{i_j}$. Consider
  the element $a = x_{i_1}^{a_1} \dotsm x_{i_s}^{a_s}$. Let $I = \set{k}{1 \leq k \leq s, a_k < 0}$.
  Note that $\prod_{j \in I} x_{i_j}^{-a_j} | f^M$ for some $M > 0$. Let $b \in A$ be such that
  \begin{equation*}
      \prod_{j \in I} x_{i_j}^{-a_j} b = f^M
  \end{equation*}
  Then
  \begin{equation*}
      a = \frac{
      \prod_{j \notin I} x_{i_j}^{a_j} b
      }{
      f^M
      }.
  \end{equation*}
  This completes the proof by taking $m = \prod_{j \notin I} x_{i_j}^{a_j} b$.
\end{proof}

\begin{theorem}
  \label{the:criOXdlb}
  Suppose $X = \mProj A$ is a multihomogeneous space defined for a
  $D$-graded integral domain $A = \bigoplus_{d \in D} A_d$ generated by
  homogeneous elements $x_1, \dotsc, x_r$ over $A_0$. 
  Moreover assume that $A_0$ is a field and $A^{\times} = A_0^{\times}$. Let
  $d \in D_f$ (see lemma \ref{lem:canchsem}) for every relevant element $f \in A$. Then
  $\strshf{X}(d)$ is a line bundle.
\end{theorem}
\begin{proof}
  By lemma \ref{lem:dplmonlX}, we can consider an open cover of $X$ given by
  relevant monomials. Fix a $d$ such that $d \in D_f$ for all relevant
  $f$. And fix an $f$ which is a relevant monomial. On $D_+(f)$,
  \begin{equation*}
    \restrict{\strshf{X}(d)}{D_+(f)} = \widetilde{A_f(d)_{(0)}} =
    \widetilde{\bigl(A(d)\bigr)_{(f)}}.
  \end{equation*}
  by lemma \ref{lem:cohshvms}\ref{ite:MDpfMftd}. We claim that $A(d)_{(f)}
  \cong A_{(f)}$. Note that $1 \in A(d)$ has degree $-d$, which belongs to
  $D_f$ by hypothesis.  Thus by lemma \ref{lem:canchsem}, we can find an
  $m$ such that $m \divides f^N$ for some $N$ and $\deg (m / f^k) = -d$ for
  some $k$. This implies $m / f^k$ is invertible in $A_f$ and $\deg f^k / m
  = d$. Now it is evident that for any element of the form $\prod_{i=1}^r
  x_i^{a_i} / f^{\nu}$ in $A(d)_{(f)}$,
  \begin{equation*}
    \deg_{A(d)} \frac{\prod_{i=1}^r x_i^{a_i}}{f^{\nu}} = 0 \iff \deg_A
    \frac{\prod_{i=1}^r x_i^{a_1}}{f^{\nu}} = d \iff \deg_A
    \frac{\prod_{i=1}^r x_i^{a_1}}{f^{\nu}} \frac{m}{f^k} = 0
  \end{equation*}
  and thus $\frac{\prod_{i=1}^r x_i^{a_1}}{f^{\nu}} \frac{m}{f^k} \in
  A_{(f)}$. Since $m / f^k$ is invertible in $A_f$, this gives an
  isomorphism of $A_{(f)}$-modules. This proves that $\strshf{X}(d)$ is a
  line bundle.
\end{proof}

\begin{example}
  In case of a weighted projective space, $P = \Proj \C[x_0, \dotsc, x_n]$
  with $\deg x_i = d_i$, theorem \ref{the:criOXdlb} reduces to saying
  $\strshf{P}(d)$ is a line bundle if and only if $d$ is divisible by each
  of the $d_i$'s. This is well known [Delorme, remark 1.8].
\end{example}

\section{A relation between a Multihomogeneous space and a $T$-variety}
\label{sec:relnmhst}

To study the relationship, we need a couple of assumptions. We shall explore
them one by one.

\begin{assumption}
  \label{ass:MHSvTvar}
  Let $D \cong \Z^r$ for a natural number $r$ and suppose $A = \bigoplus_{d \in
  D} A_d$ be a multigraded, noetherian, integral domain such that $A_0 = k$, where $k$ is
  an algebraically closed field of characteristic $0$.
\end{assumption}

\begin{assumption}
  In this section, $Y$ always refers to the vareity constructed in equation \ref{equ:constrcY}
  where, following the notation in assumption \ref{ass:MHSvTvar}, $\Spec A$ is considered as
  a $T$-variety under the action of $\Spec k[D]$.
\end{assumption}

\begin{lemma}
  \label{lem:exrelelt}
  Suppose $\Lambda$ is the GIT fan (see definition \ref{def:ormndetc})
  associated to the $T = \Spec k[D]$ action on $X = \Spec A$ induced by the
  $D$-grading. Suppose $\lambda$ is a full-dimensional cone in the quasi-fan
  $\Lambda$. Then there exists $u \in \relint \lambda$ such that $A_u$
  contains a relevant element.
\end{lemma}
\begin{proof}
  By the definition of a quasi-fan, each of the rays $\rho \in \lambda(1)$ is also
  an orbit cone and hence there exists an $u_{\rho} \in \rho \cap D$ such that
  $A_{u_{\rho}} \neq \setl{0}$.

  Since $\lambda$ is full dimensional, $\sizeof{\lambda(1)} \geq \dim
  \lambda$ and hence ($\lambda$ be a strongly convex polyhedral cone)
  $\set{u_{\rho}}{\rho \in \lambda(1)}$ is a spanning set of $D$ over
  $\Q$. Choose a homogeneous $f_{\rho} \in A_{u_{\rho}}$ for each $\rho$ and
  consider $f = \prod_{\rho \in \lambda(1)} f_{\rho}$.

  We claim that $f$ is relevant. This follows as once $f$ is inverted, the
  degrees of units in $A_{f}$ contains $\set{\pm u_{\rho}}{\rho \in
  \lambda(1)}$ and hence $[D : D_f] < \infty$, where $D_f$ is defined in
  the statement of theorem \ref{the:criOXdlb}.
\end{proof}

\begin{theorem}
  \label{the:birYmHSA}
  Under the assumption \ref{ass:MHSvTvar}, the torus $T = \Spec k[D]$ acts 
  on $X = \Spec A$ giving $X$ a structure of
  a $T$-variety which, suppose, is represented by $(Y, \mathfrak{D})$.
  Then $Y$ and $\mProj A$ are birational.
\end{theorem}

\begin{proof}
  Let $\Lambda$ be the GIT fan and $\lambda$ be a cone of maximal dimension.
  Choose a relevant $f$ using lemma \ref{lem:exrelelt} such that $\deg f \in
  \relint \lambda$. Suppose $u = \deg f$. Note that $\Spec A_f
  \hookrightarrow \Spec A$ is a $T$-equivariant embedding.
  On the other hand, consider $\smstbl{X}{u} \cap \Spec A_f$. Clearly
  both being open irreducible subsets of $\Spec A$, they are birational. Now
  the result follows from the following commutative diagram:
  \begin{equation*}
    \begin{tikzcd}
      &
      X_{ss} \arrow[d]  &
      X' = \Spec A_f \arrow[l, hook] \arrow[d] \arrow[r, hook] &
      X = \Spec A \arrow[d] \\
      Y \arrow[r] &
      Y_{\lambda} &
      U \arrow[l, hook] \arrow[r, hook]
      & \mProj A
    \end{tikzcd}
  \end{equation*}
  where the first two vertical maps are geometric quotients (by remark
  \ref{rem:frelgmqt}). The rightmost vertical map restricted to the complement of the irrelevant subscheme is a geometric quotient. Note that $Y \longrightarrow Y_{\lambda}$ is birational
  follows from \cite[lemma 6.1]{ah:affinetvar}. This
  proves that $Y$ and $\mProj A$ are birational.
\end{proof}

In the rest of this section, we shall explore conditions under which they
become isomorphic.

\begin{remark}
  It is not always true that $Y$ and $\mProj A$ considered above are isomorphic. For example, take a divisorial variety which does not admit an ample line bundle, but does admit a family of ample line bundles. Such a variety corresponds to a multihomogeneous space which is not projective. But the corresponding $Y$ will be projective by construction.
\end{remark}

\begin{assumption}
  \label{ass:genbyray}
  Suppose $\lambda = \omega$, i.e.\ the GIT fan contains only one full dimensional
  cone and its faces. Assume that $A$ is generated by $\bigcup_{u \in R}
  A_u$ where $R = \bigcup_{\rho \in \lambda(1)} \rho$.
\end{assumption}

\begin{proposition} \label{pro:isomYmPrA}
  Assume \ref{ass:MHSvTvar} and \ref{ass:genbyray}. Assume that $\omega$ is simplicial and 
  $A$ is generated by $\set{f_{\rho}}{\rho \in \omega(1)}$ such 
  that $\deg f_{\rho} \in \rho \cap D$. Then $Y$, as constructed in equation \ref{equ:constrcY}, and
  $\mProj A$ are isomorphic. In fact, both of them are projective.
\end{proposition}

\begin{proof}
  Under the given conditions, there exists a collection of relevant monomials
  $\prod_{\rho \in \omega(1)} f_{\rho}$ which have degree $u = nu'$ where $u' = \sum_{\rho} u_{\rho}$,
  $n \in \N$ and
  \[
  \mProj A = \bigcup D_+\left(\prod_{\rho \in \omega(1)} f_{\rho}\right)
  \]
  Consider $A_{(u)} = \bigoplus_{n \geq 0} A_{nu}$. It is generated by $A_u = (A_{(u)})_1$.
  Therefore, $\mProj A = \Proj A_{(u)} \cong Y$ (see \cite[6.1]{ah:affinetvar}).
\end{proof}

\begin{remark}
  In the special case when $A = k[X_1, \dotsc, X_n]$ with $\deg X_i \in \Z^d$, the affine space becomes a $T$-variety with the action of a $d$-dimensional torus. Assume that this action is effective. Then then we know that the $Y$ one gets from the description of the $T$-variety is normal and projective. It is difficult to characterize these further.
\end{remark}

\begin{corollary}
  The hypothesis of proposition \ref{pro:isomYmPrA} holds if and only if $\mProj A$ is a product of weighted projective spaces. Thus,
  $Y$ and $\mProj A$ constructed above are isomorphic if and only if $\mProj A$ is a product of weighted projective spaces.
\end{corollary}

\begin{proof}
  In the case of projective spaces and weighted projective spaces, the weight cone is the only full dimensional cone in the GIT fan. Also, if $X$ and $Y$ are varieties where the weight cones are the only full dimensional cones in their GIT fans, then the same is true for $X \times Y$.
  
  The other direction follows easily.
\end{proof}


We can not weaken the hypothesis of the above proposition \ref{pro:isomYmPrA}. Here is an example of an affine toric variety $X$ and a subtorus $T$ such that corresponding varieties $Y$ and $\mProj A$, where $A$ is the algebra of global sections of $X$, are not isomorphic. 
\begin{example}[\cite{ah:affinetvar}, example 11.1]
Take the affine toric variety $X = k^4$ associated to the canonical cone $\delta:= (\Z_{\geq 0})^4$ in $N_X = \Z^4$ and consider the subtorus $T := {k^{*}}^2$ action on $X$ given in standard coordinates by the embedding $t =(t_1,t_2) \rightarrow (t_1^4, t_1^3,t_2,t_1^{12}t_2^{-1})$. Then we have the following short exact sequence of lattices:

\begin{equation*}
    0 \xrightarrow{} N_T \xrightarrow{F} N_X \xrightarrow{P} N_Y \xrightarrow{} 0,
\end{equation*}
where $N_T$ is the lattice of one parameter subgroups of $T$ and $N_Y := N_X / N_T$ is the quotient lattice. 
We shall also consider a section $s \colon N_X \longrightarrow N_T$. The linear maps are described by
\begin{equation*}
           F = \begin{bmatrix}
            4 & 0\\
            3 & 0 \\
            0 & 1 \\
            12 & -1
           \end{bmatrix},
           \quad
           P = \begin{bmatrix}
           3 & 0 & -1 & -1\\
           0 & 4 & -1 & -1
           \end{bmatrix}
           \quad \text{and} \quad
           s = 
           \begin{bmatrix}
             1 & -1 & 0 & 0 \\
             0 & 0 & 1 & 0
           \end{bmatrix}.
\end{equation*}

Let $\Sigma_Y$ be the coarsest fan in $(N_Y)_{\Q}$ generated by $P(\delta_0)$ where $\delta_0$ are faces of $\delta$. The maximal cones of $\Sigma_Y$ are given by 
\begin{equation*}
    \sigma_1 = \langle (1,0), (0,1) \rangle,  \hspace{5pt} \sigma_2 = \langle (0,1), (-1,-1) \rangle \text{ and } \hspace{5pt} \sigma_3 = \langle (-1,-1), (1,0) \rangle.
\end{equation*} 
Then the toric variety $Y$ is $\mathbb{P}^2$ and there exist a pp-divisor $\plD$ over $Y = \mathbb{P}^2$ such that the $T$-variety $(X,T)$ is represented by the pair $(Y,\plD)$. 

Now the algebra of global sections $A = k[x_1, x_2, x_3, x_4]$ of $X$ has  a gradation by $M_T = \Z^2$ given by the $\text{deg}$ map in the following short exact sequence 
\begin{equation*}
    0 \xrightarrow{} M_Y \xrightarrow{\tilde{P}} M_X \xrightarrow{\tilde{F}} M_T \xrightarrow{} 0,
\end{equation*}
where 
\begin{equation*}
          \tilde{P} = \begin{bmatrix}
           3 & 0\\
           0 & 4\\
           -1 & -1\\
           -1 & -1
           \end{bmatrix}
           \quad \text{ and }\quad
            \tilde{F} = \begin{bmatrix}
            4 & 3 & 0 & 12\\
            0 & 0 & 1 & -1 \\
           \end{bmatrix}.
\end{equation*}
Let $I = \{1,2,3,4\}$ be an index set. Then $\deg x_1 = (4,0), \deg x_2 = (3,0), \deg x_3 = (0,1) \text{ and } \deg x_4 = (12,-1)$ in $M_T$. Let $\text{pr}_i:M_X \rightarrow \Z, i \in I$ be the projections and $\rho_i := \text{pr}_i|_{M_Y} \in N_Y$.
Then we have four rays $\rho_1 = (1,0)\R$, $\rho_2 = (0,1)\R$ and $\rho_3 = \rho_4 =(-1,-1)\R$ generated by primitive vectors. Then by remark \ref{rem:relconecorr}, a monomial $f=x_ix_j \in A$ where $i, j \in I$ is relevant if and only if the cone $\sigma_f =\langle \rho_i : i \in I \text{ and }x_i \nmid f \rangle$ is simplicial. Therefore, One can compute that
\begin{align*}
    \mProj A & = \bigcup_{f=x_ix_j \text{ relevant}} D_{+}(f) \\
    & = D_{+}(x_3x_4) \cup D_{+}(x_1x_3) \cup D_{+}(x_2x_3) \cup D_{+}(x_1x_4) \cup D_{+}(x_2x_4)
\end{align*}
 Note that
 \begin{align*}
     Y = \mathbb{P}^2 & = D_{+}(x_3x_4) \cup D_{+}(x_1x_3) \cup D_{+}(x_2x_3) \\
   &=  D_{+}(x_3x_4) \cup D_{+}(x_1x_4) \cup D_{+}(x_2x_4) 
 \end{align*}
 
 Therefore the multihomogeneous space $\mProj A$ is union of two copies of $\mathbb{P}^2$ glued along open subcheme $D_{+}(x_3x_4)$.
However the canonical map in \ref{the:birYmHSA} identifies $Y$ with either $D_{+}(x_3x_4) \cup D_{+}(x_1x_3) \cup D_{+}(x_2x_3)$ or $ D_{+}(x_3x_4) \cup D_{+}(x_1x_4) \cup D_{+}(x_2x_4)$ in $\mProj A$. And hence the map in \ref{the:birYmHSA} is not an isomorphism. The weight cone $\omega$, generated by $(0,1)$ and $(12,-1)$, is simplicial. The isomorphism fails to hold because the cone $\omega$ is not a GIT cone.
\end{example}
\bibliographystyle{plain}
\bibliography{references}

\def\cprime{$'$}
\begin{thebibliography}{10}

\bibitem{ah:affinetvar}
Klaus Altmann and J{\"u}rgen Hausen.
\newblock Polyhedral divisors and algebraic torus actions.
\newblock {\em Math. Ann.}, 334(3):557--607, 2006.

\bibitem{ahs:gentvar}
Klaus Altmann, J{\"u}rgen Hausen, and Hendrik S{\"u}ss.
\newblock Gluing affine torus actions via divisorial fans.
\newblock {\em Transform. Groups}, 13(2):215--242, 2008.

\bibitem{aipsv:geomtvar}
Klaus Altmann, Nathan~Owen Ilten, Lars Petersen, Hendrik S\"u\ss, and Robert
  Vollmert.
\newblock The geometry of {$T$}-varieties.
\newblock In {\em Contributions to algebraic geometry}, EMS Ser. Congr. Rep.,
  pages 17--69. Eur. Math. Soc., Z\"urich, 2012.

\bibitem{MR3307753}
Ivan Arzhantsev, Ulrich Derenthal, J\"{u}rgen Hausen, and Antonio Laface.
\newblock {\em Cox rings}, volume 144 of {\em Cambridge Studies in Advanced
  Mathematics}.
\newblock Cambridge University Press, Cambridge, 2015.

\bibitem{bh:gittoric}
Florian Berchtold and J{\"u}rgen Hausen.
\newblock G{IT} equivalence beyond the ample cone.
\newblock {\em Michigan Math. J.}, 54(3):483--515, 2006.

\bibitem{breshr:ampfamilies}
Holger Brenner and Stefan Schr\"{o}er.
\newblock Ample families, multihomogeneous spectra, and algebraization of
  formal schemes.
\newblock {\em Pacific J. Math.}, 208(2):209--230, 2003.

\bibitem{MR704986}
Igor Dolgachev.
\newblock Weighted projective varieties.
\newblock In {\em Group actions and vector fields ({V}ancouver, {B}.{C}.,
  1981)}, volume 956 of {\em Lecture Notes in Math.}, pages 34--71. Springer,
  Berlin, 1982.

\bibitem{MR2020670}
Hubert Flenner and Mikhail Zaidenberg.
\newblock Normal affine surfaces with {$\mathbb{C}^\ast$}-actions.
\newblock {\em Osaka J. Math.}, 40(4):981--1009, 2003.

\bibitem{hartshorne:alggeom}
Robin Hartshorne.
\newblock {\em Algebraic Geometry}.
\newblock Springer, 1977.

\bibitem{MR4033823}
Ryo Kanda.
\newblock Non-exactness of direct products of quasi-coherent sheaves.
\newblock {\em Doc. Math.}, 24:2037--2056, 2019.

\bibitem{keknmusd:toroidalembeddingsi}
G.~Kempf, Finn~Faye Knudsen, D.~Mumford, and B.~Saint-Donat.
\newblock {\em Toroidal embeddings. {I}}.
\newblock Lecture Notes in Mathematics, Vol. 339. Springer-Verlag, Berlin-New
  York, 1973.

\bibitem{MR1470147}
D.~A. Timash\"{e}v.
\newblock Classification of {$G$}-manifolds of complexity {$1$}.
\newblock {\em Izv. Ross. Akad. Nauk Ser. Mat.}, 61(2):127--162, 1997.

\bibitem{zanchetta:embdivsch}
Ferdinando Zanchetta.
\newblock Embedding divisorial schemes into smooth ones.
\newblock {\em J. Algebra}, 552:86--106, 2020.

\end{thebibliography}
\end{document}